\documentclass[]{article}
\usepackage[english]{babel}
\usepackage{lineno}
\usepackage{graphicx,multicol}
\usepackage{epic,eepic,epsfig}
\usepackage{amssymb}
\usepackage{color}
\usepackage{float}
\usepackage{amsmath}
\usepackage{amsthm}
\usepackage{todonotes}
\theoremstyle{plain}
\newtheorem{theorem}{Theorem}[section]

\newtheorem{proposition}[theorem]{Proposition}
\newtheorem{claim}{Claim}[theorem]

\newtheorem{corollary}[theorem]{Corollary}
\newtheorem{lemma}[theorem]{Lemma}

\theoremstyle{definition}

\newtheorem{problem}[theorem]{Problem}
\newtheorem{conjecture}[theorem]{Conjecture}

\newtheorem{question}[theorem]{Question}

\newenvironment{subproof}{\par\noindent {\it Proof of claim}.\ }{\hfill$\lozenge$\par\vspace{11pt}}

\newcommand{\jbj}[1]{{#1}}

\newcommand{\FH}[1]{{#1}}
\begin{document}

\title{\FH{Low chromatic spanning sub(di)graphs with prescribed degree or connectivity properties}\thanks{Research supported by the Independent
    Research Fund Denmark under grant number DFF 7014-00037B and by Agence Nationale de la Recherche under research grant ANR DIGRAPHS ANR-19-CE48-0013-01.}}
\author{J. Bang-Jensen\thanks{Department of Mathematics and Computer
    Science, University of Southern Denmark, Odense, Denmark (email:
    jbj@imada.sdu.dk). Part of this work was done while the author was visiting INRIA Sophia Antipolis. Hospitality and financial support is gratefully acknowledged. Ce travail a b\'en\'efici\'e d'une aide du gouvernement fran\c{c}ais, 
g\'er\'ee par l'Agence Nationale de la Recherche au titre du projet Investissements d'€™Avenir UCAJEDI portant la r\'ef\'erence no ANR-15-IDEX-01.}\and F. Havet\thanks{CNRS, Universit\'e
    C\^ote d'Azur, I3S and INRIA, Sophia Antipolis, France (email:
    frederic.havet@inria.fr).}\and M. Kriesell\thanks{Department of Mathematics, Technische Universit\"at Ilmenau, Germany (email: matthias.kriesell@tu-ilmenau.de)}\and
A. Yeo\thanks{Department of
    Mathematics and Computer Science, University of Southern Denmark,
    Odense, Denmark and  Department of Mathematics, University of Johannesburg, Auckland Park, 2006 South Africa (email: yeo@imada.sdu.dk).}}

\maketitle

\begin{abstract}
  
Generalizing well-known results of Erd\H{o}s and Lov\'asz, we show that every graph $G$ contains a spanning $k$-partite subgraph $H$ with $\lambda{}(H)\geq \lceil{}\frac{k-1}{k}\lambda{}(G)\rceil$, where $\lambda{}(G)$ is the  edge-connectivity of $G$. In particular, together with a well-known result due to Nash-Williams and Tutte, this  implies that
every $7$-edge-connected graphs contains a spanning bipartite graph whose edge set decomposes into two edge-disjoint spanning trees. We show that this is best possible as it does not hold for infintely many $6$-edge-connected graphs.

  For directed graphs, it was shown in \cite{bangJGT92} that there is no $k$ such that every $k$-arc-connected digraph has a spanning strong bipartite subdigraph. We prove that every strong digraph has a spanning strong 3-partite  subdigraph and that every strong semicomplete digraph on at least 6 vertices contains a spanning strong bipartite subdigraph.
\jbj{We generalize this result to higher connectivities by proving}  that, for every positive integer $k$, every $k$-arc-connected digraph contains a spanning $(2k+1$)-partite subdigraph which is $k$-arc-connected and this is best possible.  

A conjecture in \cite{kreutzerEJC24}  implies that every digraph of minimum out-degree $2k-1$ contains a spanning $3$-partite subdigraph with minimum out-degree at least $k$. We prove that the bound $2k-1$ would be best possible by providing an infinite class of digraphs with minimum out-degree $2k-2$ which do not contain any spanning $3$-partite subdigraph in which all out-degrees are at least $k$.
 We also prove that every digraph of minimum semi-degree at least $3r$ contains a spanning $6$-partite subdigraph in which every vertex has in- and out-degree at least $r$.
 
 \smallskip
 
   \noindent{}{\bf Keywords:} Edge-disjoint spanning trees; edge-connectivity; arc-connectivity; strong connectivity; bipartite graph; majority colouring; semicomplete digraph.
  
\end{abstract}

\section{Introduction}
One of the things that many courses on graph theory contain is the following fact, first observed by  Erd\H{o}s~\cite{erdosIJM3}.

\begin{proposition}[Erd\H{o}s~\cite{erdosIJM3}]
\label{prop:bipSp}
Every graph $G=(V,E)$ has a spanning bipartite subgraph $H$ such that $d_H(v)\geq \frac{1}{2}d_G(v)$ for every $v\in V$.
\end{proposition}
 It is easy to show that a spanning bipartite  subgraph with the maximum number of edges \jbj{has the desired property}. 
Finding such a graph is the {\sc Max-Cut} problem, which is well-known to be NP-complete~\cite[problem GT25]{garey1979}.  However, it is easy to construct a spanning bipartite subgraph $H$ such that $d_H(v)\geq \frac{1}{2}d_G(v)$ for every $v\in V$: start from an arbitrary spanning bipartite subgraph and then successively move a vertex to the other side if this increases the number of edges in the resulting spanning bipartite subgraph. When no more vertices can be moved, we have the desired bipartite subgraph $H$. 

It is also well-known that Proposition \ref{prop:bipSp}, as well as its constructive proof,
can be easily generalized to $k$-partite  subgraphs\jbj{, yielding the following observation due to Lov\'asz.}
\begin{proposition}[Lov\'asz~\cite{lovaszSSMH1}]
\label{prop:kpart}
For every integer $2\leq k\leq n$, every graph $G=(V,E)$ on $n$ vertices  contains a spanning $k$-partite graph $H$ satisfying $d_H(v)\geq \lceil{}\frac{k-1}{k}d_G(v)\rceil$ for every $v\in V$. In particular, $G$ has a spanning $k$-partite subgraph with at least $\lceil{}\frac{k-1}{k}|E|\rceil$ edges.
\end{proposition}

The bound in Proposition \ref{prop:kpart} can be improved and a 
 number of papers have dealt with estimating the maximum number of edges in a spanning $k$-partite subgraph, see e.g. \cite{andersenAC16,hofmeisterRSA9}.

\medskip

It is perhaps less known that Proposition \ref{prop:bipSp} can be strengthened to the result that every graph $G$ contains a spanning bipartite subgraph whose edge-connectivity is at least half of that of $G$. More generally, Proposition~\ref{prop:kpart}
can be strengthened into the following theorem. 
\begin{theorem}
\label{thm:maxcut}
Let $G$ be a graph and  $2\leq k\leq |V(G)|$ be an integer.
There is a spanning $k$-partite subgraph $H$ of  $G$ such that $\lambda(H)\geq \lceil{}\frac{k-1}{k}\lambda(G)\rceil$.
\end{theorem}

\noindent We have not been able to find this theorem in the literature, although it may already be known. We include its proof in Section~\ref{sec:Ecsec}. 
We also show how such a subgraph $H$ can be constructed in polynomial time.

\medskip

A {\bf 2T-graph} is a graph whose edge set decomposes into two edge-disjoint spanning trees.
The following theorem, due to Nash-Williams and Tutte, shows that highly edge-connected graphs have many edge-disjoint trees, in particular, every $4$-edge-connected graph contains a spanning 2T-subgraph.

\begin{theorem}[Nash-Williams~\cite{nashwilliamsJLMS36}, Tutte~\cite{tutteJLMS36}]
\label{thm:2k}
Every $2k$-edge-connected graph has $k$ edge-disjoint spanning trees.
\end{theorem}
The chromatic number of a 2T-graph is at most $4$ because it is 3-degenerate.
Observe that Theorem \ref{thm:maxcut} and Theorem \ref{thm:2k} imply the existence of  spanning 2T-subgraphs with \jbj{chromatic number smaller than 4} in graphs with high edge-connectivity.  

\begin{corollary}
\label{cor:5ec2T3part}
Let $G$ be a graph.
\begin{itemize}
\item[(a)] If $\lambda(G)\geq 5$, then $G$ contains a spanning $3$-partite  2T-subgraph.
\item[(b)]  If $\lambda(G)\geq 7$, then $G$ contains a spanning bipartite 2T-subgraph.
\end{itemize}
\end{corollary}

In Section~\ref{sec:bip2Tsec}, we show that the edge-connectivity condition $\lambda(G)\geq 7$ in Corollary~\ref{cor:5ec2T3part}~(b) is best possible by \jbj{constructing} infinitely many 6-edge-connected graphs with no spanning bipartite 2T-subgraph.
Whether the edge-connectivity condition $\lambda(G)\geq 5$ in Corollary~\ref{cor:5ec2T3part}~(a) is best possible is still open.

\medskip

Remarkably, the situation is completely different for digraphs. There is no directed analogue to Proposition \ref{prop:bipSp}.
Thomassen \cite{thomassenEJC6} proved that, for every positive integer $k$, there is a $k$-out-regular digraph $H_k$
with no  even directed  cycle, and thus with no spanning bipartite subdigraph with minimum out-degree at least $1$.
Furthermore, Bang-Jensen et al.~\cite{bangJGT92}  showed that no degree of arc-connectivity guarantees the existence of a spanning strong bipartite subdigraph.
\begin{theorem}[Bang-Jensen et al.~\cite{bangJGT92}]
\label{thm:nobipstrong}
For every positive integer $k$, there exists a $k$-arc-connected digraph $D$ which has no spanning bipartite subdigraph with minimum semi-degree at least 1. In particular, $D$ has no spanning strong bipartite subdigraph.
\end{theorem}
Moreover, Bang-Jensen et al.~\cite{bangJGT92} proved that it is NP-complete to decide whether a digraph has a spanning strong bipartite subdigraph.


On the other hand, Alon \cite{alonCPC15} pointed out that every digraph has a 3-partition $(V_1,V_2,V_3)$ such that $\Delta^+(D\langle V_i\rangle)<\Delta^+(D)$ for $i\in [3]$, where $\Delta^+(D)$ is the maximum out-degree of $D$. 
Furthermore, it is easy to show, see Proposition \ref{prop:3col-sub}, that every strong digraph has a spanning strong 3-partite  subdigraph.  We also prove that every strong semicomplete digraph on at least 6 vertices contains a spanning strong bipartite subdigraph.

In Section~\ref{sec:higharccon}, we study low chromatic spanning subdigraphs of highly arc-connected digraphs. We show that every $k$-arc-connected digraph has a spanning $(2k+1)$-partite subdigraph which is also $k$-arc-connected and this is best possible as shown by the regular tournaments on $2k+1$ vertices. We also provide an infinite family of $k$-arc-connected digraphs for which every spanning $k$-arc-connected subdigraph has chromatic number at least $2k+1$.

The above-mentioned result of Alon does not say anything about neighbours of vertices with out-degree smaller than $\Delta^+(D)$. In \cite{kreutzerEJC24}, Kreutzer et al. defined a {\bf majority colouring} of a digraph as a vertex colouring $c$ such that at least half of the  out-neighbours of every vertex $v$ have a colour different from $c(v)$.
They proved that every digraph has a majority $4$-colouring and conjectured that every digraph has  a majority $3$-colouring. 

\begin{conjecture}[Kreutzer et al.~\cite{kreutzerEJC24}]
\label{conj:3majoritycol}
  Every digraph has a majority $3$-colouring.
\end{conjecture}

To support this conjecture, they proved that it holds for a digraph $D$
such that $\delta^+(D)> 72\log{}(3|V(D)|)$, or such that $\delta^+(D)\geq \delta\geq 1200$ and
  $\Delta^-(D)\leq \frac{\exp(\delta/72)}{12\delta}$ for some $\delta$.

More generally, Kreutzer et al.~\cite{kreutzerEJC24} conjectured the following analogue of Proposition~\ref{prop:kpart}, whose particular case $r=2$ is Conjecture~\ref{conj:3majoritycol}.

\begin{conjecture}[Kreutzer et al.~\cite{kreutzerEJC24}]
  \label{conj:2k-1col}
  For every integer $r\geq 2$, every digraph $D$ contains a spanning $(2r-1)$-partite subdigraph $H$ such that every vertex $v$ satisfies $d^+_{H}(v)\geq \left\lceil\frac{r-1}{r}d^+_D(v)\right\rceil$.
  \end{conjecture}

In support to this conjecture, Knox and S\'amal~\cite{knoxEJC25} proved that  for every integer $k\geq 2$, every digraph $D$ has a spanning $k$-partite subdigraph $H$ such that every vertex $v$ satisfies $d^+_{H}(v)\geq \left\lceil\frac{k-2}{k}d^+_D(v)\right\rceil$.

 Kreutzer et al.~\cite{kreutzerEJC24} gave an example of digraphs (the regular tournaments on $2r-1$ vertices) that show that the bound $2r-1$ would be best possible in Conjecture~\ref{conj:2k-1col}. 

The truth of Conjecture \ref{conj:3majoritycol}
 would imply that every digraph of minimum out-degree at least $2k-1$ would contain a spanning $3$-partite subdigraph with minimum out-degree at least $k$. 
In Section~\ref{sec:degrees},  we prove that this would be best possible as there exist infinitely many digraphs with minimum out-degree $2k-2$ having the property that every spanning $3$-partite subdigraph has a vertex of out-degree at most $k-1$. 

We also study low chromatic subdigraphs of high minimum semi-degree and prove that every digraph of minimum semi-degree at least $3r$ contains a spanning $6$-partite subdigraph in which every vertex has in- and out-degree at least $r$ (this follows from Theorem \ref{thm:KsemiDegrees}).
  
\medskip

Finally, in Section~\ref{sec:open}, we give some final remarks and present some open questions for further research.

\section{Terminology and preliminaries}
Notation and terminology not given here is consistent with \cite{bang2009}.

\medskip

Let $G=(V,E)$ be a graph.
For two sets $X,Y$ of vertices, we denote by $d_G(X,Y)$ the number of edges with an end-vertex in $X$ and the other in $Y$.
For a subset $X\subset V$, the {\bf degree} of $X$ in $G$ is $d_G(X) = d_G(X, V\setminus X)$. 
For a vertex $v$, we abbreviate  $d_G(\{v\})$ into  $d_G(v)$.

\medskip

Let $D=(V,A)$ be a digraph. 
The {\bf underlying graph} of a digraph $D$ is the undirected graph $UG(D)=(V,E)$ where $uv\in E$ if and only if in $D$ there is an arc between $u,v$ in any direction. \jbj{A vertex $w$ is an {\bf out-neighbour}
  ({\bf in-neighbour}) of the vertex $v$ if $vw$ ($wv$) is an arc of $D$. We denote  the set of out-neighbours of a vertex $v$ by $N^+(v)$ and the set of in-neighbours of $v$ by $N^-(v)$.}

For a subset $X\subset V$ we denote by $d_D^+(X)$
(resp. $d_D^-(X)$) the number of arcs with tail (resp. head) in $X$ and head
(resp. tail) in $V\setminus X$. We call $d_D^+(X)$ (resp. $d_D^-(X)$) the {\bf out-degree}
(resp. {\bf in-degree}) of the set $X$ in $D$.
For sake of clarity, for a vertex $v$, we abbreviate  $d_D^+(\{v\})$ (resp. $d_D^-(\{v\})$) into  $d_D^+(v)$ (resp. $d_D^-(v)$).
 We also drop the subscript when the digraph is clear from the
context. \jbj{The {\bf degree} of a vertex $v$ is $d(v)=d^+(v)+d^-(v)$.} For a vertex $v$ we let $d^0_D(v)=\min\{d^+_D(v),d^-_D(v)\}$ be the {\bf semi-degree} of $v$ and we denote by $\delta^0(D)$ the minimum over all in- and
out-degrees of vertices of $D$, that is $\delta^0(D)=\min_{v\in V}d^0_D(v)$. This is also called the {\bf minimum
  semi-degree} of $D$.

\medskip

Let $G$ be a (di)graph.
For a positive integer $k$, a {\bf $\mathbf{k}$-partition} of $G$ is a partition of $V(G)$ into $k$ disjoint sets $V_1,\ldots{},V_k$.
For a subset $X$ of vertices, we denote by $G\langle X\rangle$ the sub(di)graph of $G$ {\bf induced} by $X$, that is, the sub(di)graph whose vertex set is $X$ and whose edges (arcs) are the edges (arcs) with both end-vertices in $X$. 
A {\bf bipartition} is a $2$-partition.
An {\bf independent set} in $G$ is a set of vertices that induces a sub(di)graph with no edges (arcs). 
A \jbj{ {\bf $\mathbf{k}$-colouring} of $G$ is a $k$-partition ($V_1,\ldots{},V_k)$ of $V(G)$ into independent sets}.
$G$ is {\bf $\mathbf{k}$-partite} (or {\bf $\mathbf{k}$-colourable}) if it admits a $k$-colouring \jbj{(note that we allow one or more of the sets in a $k$-partition to be empty)}.
The {\bf chromatic number} of $G$, denoted by $\chi(G)$, is the least integer $k$ such that $G$ is $k$-colourable.
If $(V_1,\ldots{},V_k)$ is a partition of the vertex set of $G)$ into disjoint subsets, then we denote by $G[V_1,\ldots{},V_k]$ the spanning $k$-partite sub(di)graph whose edges (arcs) are precisely those edges (arcs) whose end-vertices belong to different sets in the partition $(V_1,\ldots{},V_k)$.
\jbj{By a  {\bf cut} we mean} a spanning bipartite sub(di)graph of the form $G[X,V\setminus X]$ for some non-empty proper subset $X$ of $V(G)$. 
A cut is {\bf trivial} is is of the form $G[\{v\},V\setminus \{v\}]$  (or equivalently $G[V\setminus \{v\}, \{v\}]$) for some vertex $v$.

$G$ is {\bf $\mathbf{k}$-degenerate}  if  every subgraph has a vertex of degree at most $k$. It is well-known and easy to show that every $k$-degenerate (di)graph is $(k+1)$-partite.

For any two distinct vertices of $G$, we denote  by  $\lambda{}(u,v)$ the maximum number of edge-disjoint (arc-disjoint) $(u,v)$-paths in $G$. By Menger's theorem (see e.g. \cite[Section 5.4 ]{bang2009}), $\lambda{}(u,v)$ is the minimum number of edges (arcs) we need to delete from $G$ to
destroy all paths from  $u$ to $v$. The {\bf edge-connectivity} ({\bf arc-connectivity}) of $G$,
denoted by $\lambda{}(G)$, is the minimum over $\lambda{}(u,v)$ over all pairs of distinct vertices $u,v$. By Menger's theorem again this is the same as the minimum degree (out-degree) of a non-empty proper subset
of $V(G)$. 
The (di)graph $G$ is  {\bf $\mathbf{k}$-edge-connected} ({\bf $\mathbf{k}$-arc connected}) if $\lambda(G) \geq k$.
A digraph $D$ is strongly connected, or {\bf strong}, if
$\lambda{}(D)\geq 1$. \jbj{The digraph $D=(V,A)$ is {\bf $\mathbf{k}$-strong} if it has at least $k+1$ vertices and is $D-X$ is strong for every set $X\subset V$ with $|X|<k$.}

\medskip
An {\bf ear decomposition} of a digraph $D$ is a sequence
${\cal E}=(P_0,P_1,P_2,\ldots{}, P_t)$, where $P_0$ is a cycle or a vertex
and  each $P_i$ is a path, or a cycle with the following properties:
\begin{enumerate}
\item[(a)] $P_i$ and $P_j$ are arc-disjoint when $i\neq j$.
\item[(b)] For each $i=0,1,\ldots{},t$: let $D_i$ denote the digraph with vertices
 $\bigcup_{j=0}^{i}V(P_j)$ and arcs $\bigcup_{j=0}^{i}A(P_j)$.
If $P_i$ is a cycle, then it has precisely one vertex in
common with $V(D_{i-1})$. Otherwise
the end-vertices of $P_i$ are  distinct vertices
of $V(D_{i-1})$ and no other vertex of $P_i$ belongs to $V(D_{i-1})$.

\item[(c)] $\bigcup_{j=0}^{t}A(P_j)=A(D).$ 
\end{enumerate}

Each $P_i$, $0\leq i\leq t$, is called an  {\bf ear} of ${\cal E}$.  The {\bf size} of an ear $P_i$ is the number $|A(P_i)|$ of arcs in the ear. 
 An ear $P_i$ is  {\bf trivial}
 if $|A(P_i)|=1$. All other ears are {\bf non-trivial}.

The following is easy to show, see e.g. \cite[Section 5.3]{bang2009}.

\begin{theorem}
\label{eardecomp}
Let $D$ be a digraph on at least two vertices. Then
$D$ is strong if and only if it has an ear decomposition.
Furthermore, if $D$ is strong,
every cycle can be used as starting
cycle $P_0$ for an ear decomposition of $D$.
\end{theorem}

A digraph $D$ is {\bf semicomplete} if there is at least one arc between $u$ and $v$ for every pair of distinct vertices $u,v$, that is,  $UG(D)$ is a complete graph. A {\bf tournament} is a semicomplete digraph with no cycle of length 2. The following well-known result, due to Camion, was originally proved for tournaments but it is easy to see that it also holds for semicomplete digraphs.

  \begin{theorem}[Camion~\cite{camionCRASP249}]
    \label{thm:camion}
    Every strong semicomplete digraph has a hamiltonian cycle.
  \end{theorem}

  We shall also use the following generalization of Camion's theorem which, in particular, implies that every strong tournament $T$ has at least two vertices $x_1,x_2$ such that $T-x_i$ is strong.

  \begin{theorem}[Moon~\cite{moonCMB9}]
    \label{thm:moon}
    Every strong semicomplete digraph is vertex-pancyclic, that is, it has cycles of all lengths $3,4,\ldots{},n$ through each vertex.
    \end{theorem}

\section{Connectivity of low chromatic spanning  subgraphs of  graphs}\label{sec:Ecsec}

In this section, we first prove Theorem \ref{thm:maxcut} which we recall.
 
 \medskip

\noindent~{\bf Theorem~\ref{thm:maxcut}.}\ Let $G=(V,E)$ be a graph and  $2\leq k\leq |V|$ be an integer.
There is a spanning $k$-partite subgraph $H$ of  $G$ such that $\lambda(H)\geq \lceil{}\frac{k-1}{k}\lambda(G)\rceil$.

\begin{proof}

For a given partition ${\cal P}=(V_1,\ldots{},V_k)$ \jbj{of $V$}, we denote by $E_G({\cal P})$ the set of edges that go between different sets in $\cal P$, that is, $E_G({\cal P})$ is the edge set of $G[V_1,\ldots{},V_k]$. A  {\bf maximum $k$-partition} of $G$ is a $k$-partition $(V_1,\ldots{},V_k)$ of $G$ such that the number of edges in the spanning $k$-partite subgraph  $G[V_1,\ldots{},V_k]$  
 is maximized.
\FH{Let $(V_1,V_2,\ldots{},V_k)$ be a maximum $k$-partition of $G$ and
set $H=G[V_1,\ldots{},V_k]$.}

Let $X$ be a non-empty proper subset of $V$. Then $d_G(X)\geq \lambda(G)$ by the definition of $\lambda(G)$.
The bipartition $(X,V\setminus X)$ induces a bipartition of each set $V_i$ into $X_i=V_i\cap X$ and $Y_i=V_i\cap (V\setminus X)$.
Now it follows that 
\begin{equation}\label{Zdeg}
d_G(X)=\sum_{i,j\in [k]}d_G(X_i,Y_j)
\end{equation}

\begin{equation}
d_H(X)=\sum_{i\neq j\in [k]}d_G(X_i,Y_j).
\end{equation}
We claim that the later is at least $\lceil\frac{k-1}{k}d_G(X)\rceil$.\\

\FH{
We can write $d_G(X)$ and $d_H(X)$ as 
\begin{eqnarray}
d_G(X) & = & \sum_{\ell=0}^{k-1}e_{\pi_{\ell}} \label{contributions} \\
d_H(X) & = & \sum_{\ell=1}^{k-1}e_{\pi_{\ell}}
\end{eqnarray} 
where $\pi_{\ell}=(1234\ldots{}k)^{\ell}$, that is, the cyclic permutation which shifts the indices $1,2,\ldots{},k$ cyclically to the right $\ell$ times and $e_{\pi_{\ell}}=\sum_{i=1}^kd_G(X_i,Y_{\pi_{\ell}(i)})$. Hence $\pi_0$ is the trivial permutation that fixes everything and $e_{\pi_0}= \sum_{i\in [k]}d_G(X_i,Y_i)$ 
}

As ${\cal P}$ is a maximum $k$-partition, $e_{\pi_0}\leq e_{\pi_i}$ for $i\in [k-1]$, so 
(\ref{contributions}) implies that $e_{\pi_0}\leq \lfloor{}\frac{d_G(X)}{k}\rfloor$ and thus $d_H(X)\geq \lceil{}\frac{k-1}{k}d_G(X)\rceil\geq \lceil{}\frac{k-1}{k}\lambda(G)\rceil$. As $X$ was arbitrary, we conclude that 
$\lambda(H)\geq \lceil{}\frac{k-1}{k}\lambda(G)\rceil$.
\end{proof}

The statement above uses a maximum $k$-partition for convenience. It is well-known that finding such a partition is NP-hard \cite{aroraFOCS33} so the proof of Theorem \ref{thm:maxcut} does not immediately give a polynomial-time algorithm for finding the desired spanning subgraph $H$. However, the idea used in the proof can be made algorithmic. 
\begin{corollary}
There exists a polynomial-time algorithm which, given a graph $G$ and  an integer $2\leq k\leq |V(G)|$, constructs a spanning $k$-partite subgraph $H$ of  $G$ such that $\lambda(H)\geq \lceil{}\frac{k-1}{k}\lambda(G)\rceil$.
\end{corollary}
\begin{proof}
It is well-known that we can find the edge-connectivity of a given graph $G$ as well as a non-empty proper subset $W$ of its vertices satisfying that $d_G(W)=\lambda(G)$ in polynomial time (e.g. using flows, see e.g. \cite[Section 5.5]{bang2009}, or using maximum adjacency orderings \cite{nagamochiSJDM5}). We claim that we can obtain the desired subgraph $H$ using such an algorithm at most $|E|$ times.

Let $G=(V,E)$ and $2\leq k\leq |V|$ be given. Start from an arbitrary $k$-partition ${\cal P}_0=(V^0_1,\ldots{},V^0_k)$ of $V$ and let $H_0=G[V^0_1,\ldots{},V^0_k]$ be the spanning $k$-partite graph induced by ${\cal P}_0$. If $\lambda(H_0)\geq \lceil\frac{k-1}{k}\lambda(G)\rceil$, we are done, so assume this is not the case and let $X$ be a non-empty proper subset of $V$ for which $d_{H_0}(X)<\lceil\frac{k-1}{k}\lambda(G)\rceil$. Then we consider the sets $X^0_1,\ldots{},X^0_k, Y^0_1,\ldots{},Y^0_k$, where $X^0_i=V^0_i\cap X$ and $Y^0_i=V^0_i\cap (V\setminus X)$, $i\in [k]$, as in the proof of Theorem \ref{thm:maxcut}. Since we have  $\lambda(H_0)<\lceil\frac{k-1}{k}\lambda(G)\rceil$ it follows from the proof that there exists an index $q\in [k-1]$ such that 
\begin{equation}
\label{improve}
\sum_{i\in [k]}d_G(X^0_i,Y^0_i)>\sum_{i\in [k]}d_G(X^0_i,Y^0_{i+q}).
\end{equation}
 
Now consider the partition ${\cal P}_1=(V^1_1,\ldots{},V^1_k)$, where 
$V^1_i=X^0_i\cup Y^0_{i+q}$. By (\ref{improve}) we see that $|E_G({\cal P}_1)|>|E_G({\cal P}_0)|$.

As long as the  spanning $k$-partite graph induced by the partition is not the desired one, repeating the steps above, we find a $k$-partition with more edges across it than the current one.
Hence the process will stops after at most $|E|$ loops. When it stops, we necessarily have a $k$-partition inducing a digraph $H$ such that  $\lambda(H)\geq \lceil{}\frac{k-1}{k}\lambda(G)\rceil$.
\end{proof}

\begin{corollary}
  \label{cor:k+1ec}
Every graph $G$ contains a spanning $(\lambda(G)+1)$-chromatic $\lambda(G)$-edge-connected subgraph $H$. 
\end{corollary}

\FH{Note that this corollary is best possible as a complete graph $K$  has no $\lambda(K)$-chromatic $\lambda(K)$-edge-connected subgraph.}

\section{$6$-edge-connected graphs with no spanning bipartite 2T-subgraph}\label{sec:bip2Tsec}

In this section, we shall prove that there are infinitely many $6$-edge-connected graphs with no spanning bipartite 2T-subgraph. To do so, we need some preliminaries.

A bipartite graph $G$ is {\bf $\mathbf{(a,b)}$-regular} if all the vertices of one of its partite sets have degree $a$ and all vertices of the other partite set have degree $b$.
A graph $G=(V,E)$ is {\bf essentially-$\mathbf{k}$-edge-connected} if every non-trivial cut has at least $k$ edges.

\begin{proposition}
\label{prop:es6ec}
There are infinitely many $(3,5)$-regular bipartite graphs which are essentially $6$-edge-connected.
\end{proposition}
\begin{proof}
We will give an explicit construction of such a graph on $n$ vertices for every $n$ divisible by 8.
Let $H$ be a cubic bipartite essentially 4-edge-connected graph on $\frac{3}{4}n$ vertices with partite sets $A$, $C$ (each of order $\frac{3}{8}n$). Let $L$ be a cubic essentially 4-edge-connected graph on $\frac{n}{4}$ vertices disjoint from $H$; it has $\frac{3}{8}n$ edges. Now subdivide every edge of $L$ once as to obtain the graph $S$, and let $G$ be obtained from the union of $H$ and $S$ by 
identifying  the subdivision vertices of $S$ one-to-one with the vertices of $A$.  The resulting graph $G$ is bipartite with bipartition $(A,B)$, with $B = V(L) \cup C$. Moreover, every vertex in $A$ has degree 5 and every vertex in $B$ has degree 3. See Figure~\ref{fig:G16} for an example on 16 vertices.
We claim that $G$ is essentially 6-edge-connected.

\newcommand{\LineB}[1]{\draw [color=blue, line width=0.03cm] #1}
\newcommand{\LineR}[1]{\draw [color=red, line width=0.03cm] #1}
\newcommand{\Nodebox}[3]{\node (#1)  at (#2) [vertexbox] {$#3$}}
\newcommand{\Node}[3]{\node (#1) at (#2) [vertexN] {$#3$}}
\tikzstyle{vertexX}=[circle,draw, top color=gray!10, bottom color=gray!70, minimum size=14pt, scale=0.6, inner sep=0.1pt]
\tikzstyle{vertexY}=[circle,draw, top color=green!10, bottom color=green!70, minimum size=14pt, scale=0.6, inner sep=0.1pt]
\tikzstyle{vertexZ}=[circle,draw, top color=orange!10, bottom color=orange!70, minimum size=14pt, scale=0.6, inner sep=0.1pt]
\tikzstyle{vertexbox}=[rectangle, draw, minimum size=14pt, scale=0.6, inner sep=0.5pt]
\tikzstyle{vertexr}=[circle,draw, color=red, minimum size=7pt, scale=0.6, inner sep=0.5pt]
\tikzstyle{vertexb}=[circle,draw, color=blue,minimum size=7pt, scale=0.6, inner sep=0.5pt]
\tikzstyle{vertexN}=[circle,draw, minimum size=14pt, scale=0.6, inner sep=0.5pt]
\tikzstyle{vertexR}=[circle,draw, top color=red!100, bottom color=red!100, minimum size=7pt, scale=0.6, inner sep=0.5pt]
\tikzstyle{vertexB}=[circle,draw, top color=blue!100, bottom color=blue!100, minimum size=7pt, scale=0.6, inner sep=0.5pt]

\begin{figure}[H]
\begin{center}

\begin{tikzpicture}[scale=0.6]
\Nodebox{c6}{6,1}{c_6};
\Nodebox{w}{6.75,5.25}{w};
\Nodebox{c5}{4,4}{c_5};
\Nodebox{c3}{8,4}{c_3};
\Nodebox{x}{2.5,6}{x};
\Nodebox{c2}{10,8}{c_2};
\Nodebox{c1}{6,8}{c_1};
\Nodebox{c4}{2,8}{c_4};
\Nodebox{y}{9.5,6}{y};
\Nodebox{z}{4.75,5.25}{z};
\Node{a1}{6,10}{a_1};
\Node{a2}{8,7}{a_2};
\Node{a3}{10,4}{a_3};
\Node{a4}{4,7}{a_4};
\Node{a5}{2,4}{a_5};
\Node{a6}{6,3}{a_6};
\LineB{(a1) to (c1)};
\LineB{(a2) to (c2)};
\LineB{(a3) to (c3)};
\LineB{(a4) to (c4)};
\LineB{(a5) to (c5)};
\LineB{(a6) to (c6)};
\LineB{(c1) to (a4)};
\LineB{(c5) to (a4)};
\LineB{(c5) to (a6)};
\LineB{(c3) to (a6)};
\LineB{(c1) to (a2)};
\LineB{(c3) to (a2)};
\LineB{(a1) to (c4)};
\LineB{(a1) to (c2)};
\LineB{(a5) to (c4)};
\LineB{(a5) to (c6)};
\LineB{(a3) to (c2)};
\LineB{(a3) to (c6)};
\LineR{(z) to (a2)};
\LineR{(z) to (a4)};
\LineR{(z) to (a6)};
\LineR{(x) to (a1)};
\LineR{(x) to (a5)};
\LineR{(x) to (a4)};
\LineR{(y) to (a1)};
\LineR{(y) to (a2)};
\LineR{(y) to (a3)};
\LineR{(w) to (a5)};
\LineR{(w) to (a6)};
\LineR{(w) to (a3)};
\end{tikzpicture}
\caption{An essentially 6-edge-connected bipartite graph $G$. It is obtained from the blue graph $H=C_6\square K_2$ and the red graph $S$ which is  a subdivision of a $K_4$ on vertices $\{x,y,z,w\}$ by identifying the six subdivision vertices of $S$  with the vertices $\{a_1,a_2,a_3,a_4,a_5,a_6\}$ of  $H$.}\label{fig:G16}
\end{center}
\end{figure}
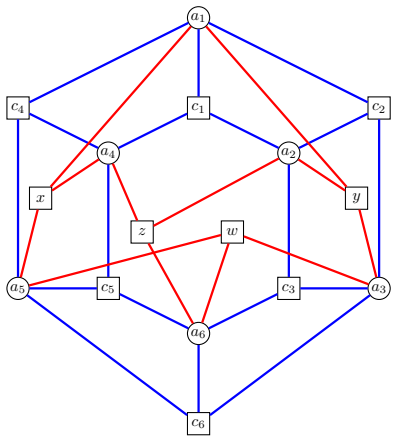

\FH{Let $\Gamma=G[X, V(G)\setminus X]$ be a non-trivial cut of $G$ with the minimum number of edges.
We shall prove that it has at least six edges.

If $G\langle X \rangle$ is not connected, then by \jbj{the minimality of $\Gamma$, $X$ consists of} two non-adjacent vertices, and so $\Gamma$ has at least six edges.
Similarly, we get the result if $G-X$ is not connected. Henceforth, we may assume that both $G\langle X \rangle$ and $G-X$ are
connected.

Assume first that $G\langle X\rangle$ contains an edge $u_1v_1$ of  $E(H)$ and $G-X$ contains an edge $u_2v_2$ of $E(H)$.
Since $H$ is essentially $4$-connected, by Menger's theorem, there are four edge-disjoint paths from $\{u_1,v_1\}$ to $\{u_2,v_2\}$ using only  edges from $E(H)$. Moreover, there are at least two edge-disjoint paths from $\{u_1,v_1\}$ to $\{u_2,v_2\}$ using only edges from $E(S)$. Each of these six paths contains a distinct edge of $E(\Gamma)$. Thus $\Gamma$ has at least six edges.

Henceforth, we may assume that one of $G\langle X\rangle$ and $G-X$ contains no edge of $H$.
By symmetry, we may assume $E(G\langle X\rangle) \cap E(H)=\emptyset$.

If $X$ contains two subdivision vertices of $S$ (which are identified to vertices of $H$), then $\Gamma$  contains the six edges
of $E(H)$ incident to those vertices. If not, then $X$ consists of a single subdivision vertex $x$ plus one or two neighbours of $x$ in $S$; if $|X)|=2$ then $E(\Gamma)$ consists of three edges from $E(H)$ and three edges from $E(S)$, and if $|X|=3$ then $E(\Gamma)$ consists of three edges from $E(H)$ and four edges from $E(S)$. In either case, $|E(\Gamma)|$ is at least 6.

This proves that the $G$ is essentially 6-edge-connected.}
\end{proof}

\begin{theorem}
\label{thm:7ecbip2T}
There are infinitely many $6$-edge-connected graphs with no spanning bipartite 2T-subgraph.
\end{theorem}
\begin{proof} We shall use the family ${\cal G}$ of essentially $6$-edge-connected graphs from Proposition \ref{prop:es6ec}. Let $G\in {\cal G}$ 
have at least 16 vertices and   $b$ vertices of degree 3 and $c$ vertices of degree 5, so that $G$ has $n = 3b = 5c$ edges.  Consider the line graph $H$ of $G$. It has $n$ vertices and is the edge-disjoint union of $b$ copies of $K_3$ and $c$ copies of $K_5$. Every cut $H[X,V\setminus X]$ contains at most 2 edges of every triangle and at most 6 edges of any $K_5$. Hence a cut has at most $2b+6c = \frac{2}{3}n+ \frac{6}{5}n = \frac{28}{15}n < 2n-2$. So $H$ has no spanning bipartite subgraph with at least $2n-2$ edges and hence cannot contain a spanning bipartite 2T-subgraph. 
Moreover, since $G$ is essentially $6$-edge-connected, \jbj{the graph} $H$ is $6$-connected.

Thus the line graphs of graphs in $\cal G$ on at least 16 vertices form an infinite   class of 6-regular 6-connected graphs which do not admit a spanning bipartite 2T-subgraph.
\end{proof}

\section{Spanning strong subdigraphs with low chromatic number}\label{sec:Arcconsec}


\begin{proposition}\label{prop:3col-sub}
Every strong digraph has a spanning strong $3$-partite subdigraph  and such a subdigraph can be constructed in polynomial time.
\end{proposition}
\begin{proof}
  Consider an ear-decomposition ${\cal E}=\{P_0,P_1,\ldots{},P_q,P_{q+1},\ldots{},P_t\}$ constructed from an arbitrary cycle $P_0$ by adding in each step $i\in [q]$ a shortest non-trivial ear and in steps $q+1,\ldots{},t$ trivial ears. Let $D_q$ be the union of $P_0,\ldots{},P_q$.  Then $D_q$ is a spanning subdigraph of $D$ and it is easy to see that its underlying graph $UG(D_q)$ is  $2$-degenerate because it follows from the choice of the paths $P_1,\ldots{},P_q$ that for each $i\in [q]$  the internal vertices of $P_i$ have in-degree and out-degree one in  $D_{i}$. 
    Hence $\chi(D_q)=\chi(UG(D_q))\leq 3$.
\end{proof}

A digraph is {\bf semicomplete bipartite} if its  underlying graph is a complete bipartite graph $K_{r,s}$ for some $r,s\geq 1$. A {\bf bipartite tournament} is a semicomplete bipartite digraph with no 2-cycles.
We now show that every strong semicomplete digraph on at least 6 vertices  has a spanning strong
  semicomplete bipartite subdigraph and hence has a spanning strong bipartite subdigraph.
For a semicomplete digraph on an even number of vertices it is trivial that strong connectivity is necessary and sufficient: By Theorem \ref{thm:camion}, every strong semicomplete digraph has a hamiltonian cycle. Hence it suffices to  colour the vertices of a hamiltonian cycle alternately by 1 and 2. If the order is odd, we need to do more work. We first consider tournaments. Let $T_5$ be the strong tournament that we obtain from a 3-cycle $abca$ by adding two new vertices $d,e$ and the arcs $de,ad,bd,cd,ea,eb,ec$ (see Figure \ref{fig:exceptions}).  It is easy to check that $T_5$ has no spanning bipartite subdigraph which is strong.

The following two known results will turn out to be very useful.

\begin{theorem}[Fraisse and Thomassen~\cite{fraisseGC3}] \label{FrTh}
Every \jbj{$r$-strong} tournament contains a hamiltonian cycle avoiding any $r-1$ arcs.
\end{theorem}

The following lemma is well-known and easy to prove, as if $S$ is a separating set in a 
regular tournament $T$, and $(A,B)$ is a partition of $V(T-S)$ such that all arcs between 
$A$ and $B$ go from $A$ to $B$, then $S$ is at least as large as $A$ and as $B$.

\begin{lemma}[Thomassen~\cite{thomassenJCT28}] \label{strongT}
Every regular tournament is \jbj{$\lceil n/3 \rceil$-strong}.
\end{lemma}

\begin{theorem}
  \label{thm:TspBT}
  Every strong tournament $T$ which is different from the $3$-cycle $C_3$ and from $T_5$ contains a spanning strong bipartite tournament.
\end{theorem} 
\begin{proof}
  Let $T=(V,A)$ be  a strong tournament on $n$ vertices.  If $n$ is even, then, as we argued above, $T$ has a spanning strong bipartite tournament. Hence we may assume that $n=2k+1$ for some $k\geq 1$.
 
A bipartition $(V_1, V_2)$ of $T$ is {\bf good} if $T[V_1,V_2]$ is strong.
We shall prove by induction on $k$ that $T$ admits a good bipartition unless $T=C_3$ or $T=T_5$.
If $k=1$, then $T=C_3$ so there is nothing to prove. Hence we may assume that $k \geq 2$ and move to the induction step.

By Theorem \ref{thm:moon},  $T$ has a vertex $z$ such that $T'=T-z$ is strong. 
Choose such a vertex $z$ so that $d^0(z)=\min\{d^+(z),d^-(z)\}$ is as small as possible. 
For a hamiltonian cycle $C=v_1v_2\ldots{}v_{2k}v_1$ in $T'$ let $(V_1^C,V_2^C)$ be the 2-partition of \jbj{$V(C)$} such that
$V_1^C=\{v_1,v_3,\ldots{},v_{2k-1}\}$ and  $V_2^C=\{v_2,v_4,\ldots{},v_{2k}\}$.
If $z$ has both an in-neighbour and an out-neighbour in $V_i^C$, then $(V_i^C,V_{3-i}^C\cup\{z\})$ is a good bipartition.
Hence we can assume the following for all hamiltonian cycles, $C$ in $T'$.

\begin{claim}\label{claimA}
$N^-(z)=V_i^C$ and $N^+(z)=V_{3-i}^C$ for some $i \in [2]$.
\end{claim}

Let $C=v_1v_2\ldots{}v_{2k}v_1$ be any hamiltonian cycle of $T'$ and, by Claim~\ref{claimA}, assume without loss of generality that $N^-(z)=V_1^C$ and $N^+(z)=V_2^C$. \jbj{Note that $d^+(z)=d^-(z)=k$.}

Suppose first that $T$ is not a $k$-regular tournament. 
Then we can relabel $V\setminus \{z\}$ such that $d^-(v_2)\neq d^+(v_2)$ \jbj{(as if $d^-(v_i)= d^+(v_i)$ for all $i\in [2k]$ then $T$ is $k$-regular since we also have $d^-(z)=d^+(z)$
).} 
Observe that  $C''=v_1zv_4\ldots{}v_{2k}v_1$ is an $(n-2)$-cycle in $T$.
 If the vertex $v_3$ has both an in-neighbour and an out-neighbour on $C''$ then $T-v_2$ is strong, contradicting the choice of $z$ (as we have concluded that $d^-(z)=d^+(z)$ above).
Hence we may asssume that $v_3$ has no in-neighbour on $C''$ and again by the choice of $z$ we may assume that  $v_2$ has no out-neighbour on $C''$. Let $T''=T-\{v_2,v_3\}$ and note that $T''$ is strong since $C''$ is a hamiltonian cycle of $T''$. 
\FH{If $T''$ has a good bipartition $(X,Y)$, then $(X\cup \{v_2\}, Y\cup\{v_3\})$ is a  good bipartition of $T$.} Hence we may assume that $T''$ has no spanning strong bipartite subdigraph. Thus, by induction, $T''$ is either $C_3$ or $T_5$. If $T''=C_3$, then it is easy to see that $T=T_5$. If $T''=T_5$, then $\{v_2,b,c,e\},\{v_3,a,d\}$ is a good partition of $V(T)$.

Suppose now that $T$ is a $k$-regular tournament, that is, $d^+(v)=d^-(v)=k$ for every vertex. 
By Lemma~\ref{strongT}, $T$ is at least \jbj{$\lceil n/3 \rceil$-strong}, so $T'=T-z$ is at least \jbj{$\left(\lceil n/3 \rceil -1 \right)$-strong}.
Let $v_i \in V_1^C$ be a vertex with as few out-neighbours in $V_2^C$ as possible.
As $N^-(z) = V_1^C$, there are $k(k-1)/2$ arcs within $V_1^C$ and $k(k-1)/2$ arcs from $V_1^C$ to $V_2^C$ (as each vertex in $V_1^C$ has out-degree $k-1$ in $T'$).
Therefore $v_i$ has at most $\lfloor (k-1)/2 \rfloor$ out-neighbours in $V_2^C$.
As $k \geq 2$ the following holds,

\[
\lfloor (k-1)/2 \rfloor \leq (k-1)/2 < (2k-2)/3 \leq \lceil (2k+1)/3 \rceil -1  = \lceil n/3 \rceil -1
\]

Therefore,  by Theorem~\ref{FrTh}, there exists a hamiltonian cycle, $C^*$, in $T'$, avoiding all arcs from $v_i$ to $V^C_2$. This means that the \jbj{successor} of $v_i$ in $C^*$ is also in $V^C_1$.
However $z$ now has an in-neighbour in $V_j^{C^*}$ (namely $v_i$) and in $V_{3-j}^{C^*}$ (the \jbj{successor} of $v_i$ on $C^*$) for some $j \in [2]$,
a contradiction to Claim~\ref{claimA} above.
\end{proof}

Denote by $C_{3,1}$ the semicomplete digraph on three vertices $x,y,z$ with arcs $xy,yz,yx,zx$.
The digraphs $T_5$, $S_{5,1}$, $S_{5,2}$, and $S_{5,3}$ are the ones depicted in Figure~\ref{fig:exceptions}.

\begin{figure}[hbtp]
    \begin{center}
      \tikzstyle{vertexB}=[circle,draw, minimum size=20pt, scale=0.6, inner sep=0.5pt]
      \begin{tikzpicture}[scale=0.6]
       
     \node (a) at (90:2cm) [vertexB] {$a$};
     \node (b) at (162:2cm) [vertexB] {$b$};
     \node (c) at (18:2cm) [vertexB] {$c$};
     \node (d) at (-54:2cm) [vertexB] {$d$};
 \node (e) at (-126:2cm) [vertexB] {$e$};
        \draw[->, line width=0.03cm] (a) to (b);
        \draw[->, line width=0.03cm] (a) to (d);
        \draw[->, line width=0.03cm] (b) to (c);
        \draw[->, line width=0.03cm] (b) to (d);
        \draw[->, line width=0.03cm] (c) to (a);
        \draw[->, line width=0.03cm] (c) to (d);
        \draw[->, line width=0.03cm] (d) to (e);
        \draw[->, line width=0.03cm] (e) to (a);
        \draw[->, line width=0.03cm] (e) to (b);
        \draw[->, line width=0.03cm] (e) to (c);
        \node () at (0,-2.6) {$T_5$};
   \end{tikzpicture}
   \hfill
  \begin{tikzpicture}[scale=0.6]
  \node (a) at (90:2cm) [vertexB] {$a$};
     \node (b) at (162:2cm) [vertexB] {$b$};
     \node (c) at (18:2cm) [vertexB] {$c$};
     \node (d) at (-54:2cm) [vertexB] {$d$};
 \node (e) at (-126:2cm) [vertexB] {$e$};
        \draw[->, line width=0.03cm] (a) to (b);
        \draw[->, line width=0.03cm] (a) to (d);
        \draw[->, line width=0.03cm] (b) to (c);
        \draw[->, line width=0.03cm] (b) to (d);
        \draw[->, line width=0.03cm] (c) to (a);
        \draw[->, line width=0.03cm] (c) to (d);
        \draw[->, line width=0.03cm] (d) to (e);
        \draw[->, line width=0.03cm] (e) to (a);
        \draw[->, line width=0.03cm] (e) to (b);
        \draw[->, line width=0.03cm] (e) to (c);
        \draw[->, line width=0.03cm] (e) to [out=-30,in=210] (d);
        \node () at (0,-2.6) {$S_{5,1}$};
 \end{tikzpicture}
   \hfill
  \begin{tikzpicture}[scale=0.6]
      \node (a) at (90:2cm) [vertexB] {$a$};
     \node (b) at (162:2cm) [vertexB] {$b$};
     \node (c) at (18:2cm) [vertexB] {$c$};
     \node (d) at (-54:2cm) [vertexB] {$d$};
 \node (e) at (-126:2cm) [vertexB] {$e$};
        \draw[->, line width=0.03cm] (a) to (b);
        \draw[->, line width=0.03cm] (a) to (d);
        \draw[->, line width=0.03cm] (b) to (c);
        \draw[->, line width=0.03cm] (b) to (d);
        \draw[->, line width=0.03cm] (c) to (a);
        \draw[->, line width=0.03cm] (c) to (d);
        \draw[->, line width=0.03cm] (d) to (e);
        \draw[->, line width=0.03cm] (e) to (a);
        \draw[->, line width=0.03cm] (e) to (b);
        \draw[->, line width=0.03cm] (e) to (c);
        \draw[->, line width=0.03cm] (b) to [out=-105, in=150] (e);
        \node () at (0,-2.6) {$S_{5,2}$};
  \end{tikzpicture}
   \hfill
  \begin{tikzpicture}[scale=0.6]       
         \node (a) at (90:2cm) [vertexB] {$a$};
     \node (b) at (162:2cm) [vertexB] {$b$};
     \node (c) at (18:2cm) [vertexB] {$c$};
     \node (d) at (-54:2cm) [vertexB] {$d$};
 \node (e) at (-126:2cm) [vertexB] {$e$};
        \draw[->, line width=0.03cm] (a) to (b);
        \draw[->, line width=0.03cm] (a) to (d);
        \draw[->, line width=0.03cm] (b) to (c);
        \draw[->, line width=0.03cm] (b) to (d);
        \draw[->, line width=0.03cm] (c) to (a);
        \draw[->, line width=0.03cm] (c) to (d);
        \draw[->, line width=0.03cm] (d) to (e);
        \draw[->, line width=0.03cm] (e) to (a);
        \draw[->, line width=0.03cm] (e) to (b);
        \draw[->, line width=0.03cm] (e) to (c);
        \draw[->, line width=0.03cm] (d) to [out=30, in=-80] (c);
        \node () at (0,-2.6) {$S_{5,3}$};
  \end{tikzpicture}      
  \end{center}
    \caption{The semicomplete digraphs on 5 vertices that have no spanning strong bipartite subdigraph.}\label{fig:exceptions}
    \end{figure}
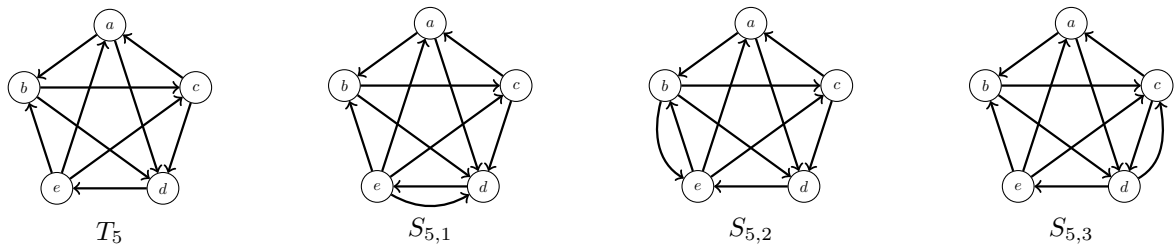

        \begin{theorem}
      \label{thm:SDstrongbip}
      Let $S$ be a strong semicomplete digraph. Then $S$ has a spanning strong semicomplete bipartite subdigraph if and only if $S$ is not isomorphic to one of $C_3,C_{3,1},T_5,S_{5,1},S_{5,2},S_{5,3}$. \jbj{In particular, every strong semicomplete digraph on at
        least 6 vertices has a spanning strong semicomplete bipartite subdigraph.}
    \end{theorem}
    \begin{proof}
      We already know that $C_3$ and $T_5$ have no spanning strong bipartite subdigraph and it is easy to check that neither does any of $C_{3,1},S_{5,1},S_{5,2},S_{5,3}$. Hence we may assume that $S$ is
      not isomorphic to one of $C_3,C_{3,1},T_5,S_{5,1},S_{5,2},S_{5,3}$.
      If $S$ has just two vertices, then it is a 2-cycle which is a strong semicomplete bipartite digraph, so we may assume that $|S|>2$.  By Theorem \ref{thm:camion}, $S$ has a hamiltonian cycle $C$. Now we obtain a spanning strong tournament $T$ of $S$ by 
      deleting  one arc of every 2-cycle while maintaining $C$ as a subdigraph. Hence it follows from Theorem \ref{thm:TspBT} that $S$ has the desired spanning strong semicomplete bipartite subdigraph unless $T$ is one of the tournaments $C_3$ or $T_5$. 
      
      Suppose first that $T=C_3$. Then, as $S$ is not one of 
      $C_3$ or $C_{3,1}$, it has at least two 2-cycles and thus clearly has a spanning strong semicomplete bipartite subdigraph. Hence we may assume that $T=T_5$. 
      
      If $S\langle \{a,b,c\}\rangle$ has a 2-cycle, then we may assume w.l.o.g. that $aba$ is a 2-cycle and now we see that $S[\{a,e\},\{b,c,d\}]$  is strong. Hence we may assume that there are no 2-cycles in  $S\langle \{a,b,c\}\rangle$. Suppose that $ded$ is a 2-cycle. Then, as $S$ is not isomorphic to $S_{5,1}$, it has another 2-cycle with one vertex in $\{d,e\}$. By reversing all arcs and renaming $a,b,c$, if necessary, we can assume that $cec$ is a 2-cycle and now $S[\{a,c,d\},\{b,e\}]$ is strong. Hence we may assume that there is no 2-cycle on $\{d,e\}$. Since $S$ is not isomorphic to one of $S_{5,2},S_{5,3}$, it must have at least two 2-cycles with one vertex in $\{d,e\}$ and the other in $\{a,b,c\}$ (by the assumptions made so far). If $S$ contains two disjoint 2-cycles, each with one vertex in $\{d,e\}$, then we may assume that $aea$ and $cdc$ are 2-cycles and thus $S[\{c,e\},\{a,b,d\}]$ is strong. If $S$ has two 2-cycles both containing $e$ and with the other end-vertices in $\{a,b,c\}$, then we may assume that these are  $a,c$ and now $S[\{b,e\},\{a,c,d\}]$ is strong. Finally, if $S$ has two 2-cycles both containing $d$ and with the other end-vertices in $\{a,b,c\}$, then we may assume that $ada$ and $cdc$ are 2-cycles and thus $S[\{b,d\},\{a,c,e\}]$ is strong.
\end{proof}

\vspace{2mm}

Kim et al. \cite{kimSJDM30} proved that for every natural number $k$ there exists a function $f(k)$ such that every
\jbj{$f(k)$-strong} tournament has a bipartition $(V_1,V_2)$ such that each of the digraphs $D\langle V_1\rangle$, $D\langle V_2\rangle$, $D[V_1,V_2]$ are $k$-connected.\\

The arc-connectivity analogue does not hold even for $k=1$.

\begin{proposition}\label{prop:tour}
 There exists no natural number $K$  such that every $K$-arc-connected tournament admits a bipartition $(V_1,V_2)$ such that each of the digraphs $D\langle V_1\rangle,D\langle V_2\rangle,D[V_1,V_2]$ are strong.
\end{proposition}
\begin{proof}
  Let $K$ be given and let $T$ and $T'$ be two $K$-arc strong tournaments. Form the tournament $D$ by adding all arcs from $V(T)$ to $V(T')$, all arcs from $V(T')$ to $x$ and all arcs from $x$ to $V(T)$, where $x$ is a new vertex.
  Then $D$ is $K$-arc-connected but there is no bipartition $(V_1,V_2)$ such that
  each of the digraphs $T\langle V_i\rangle$ are strong and $\delta^0{}(T[V_1,V_2])\geq 1$:
  Suppose that $(V_1,V_2)$ is such a bipartition.  W.l.o.g. $x\in V_1$.  We cannot  have that $V_2\cap V(T),V_2\cap V(T')$ are both non-empty, since then $D\langle V_2\rangle$ would not be strong. Suppose $V(T')\subset V_1$. Then there is no arc from $V(T')$ to $V_2$. Similarly we get a contradiction if  $V(T')\subset V_2$.
  
  \end{proof}

It was shown in \cite{bangDAM146} that the existence of a bipartition $(V_1,V_2)$ such that each of the digraphs $D\langle V_1\rangle,D\langle V_2\rangle$ are strong can be checked in polynomial time \jbj{when the input is a semicomplete digraph}.

\begin{question}
  What is the complexity of deciding whether a given semicomplete digraph $D$ has a bipartition $(V_1,V_2)$ such that each of the digraphs $D\langle V_1\rangle,D\langle V_2\rangle ,D[V_1,V_2]$ is strong?
\end{question}

\section{Highly arc-connected subdigraphs of low chromatic number}\label{sec:higharccon}

In this section, digraphs may have multiple arcs.
We shall use the following three results due to Mader.
\begin{theorem}[Mader~\cite{maderEJC3}]
  \label{thm:madersplit}
\FH{Let $D=(V,A\cup F)$ be a digraph, where $F$ is the set of arcs incident to a special vertex $s$.  Suppose  $d^-(s)=d^+(s)$ and $\lambda(D)\geq k$.} Then there exists a pairing $(u_1s,sv_1),\ldots{},(u_{d^-(s)},sv_{d^-(s)})$ of the arcs incident to $s$ such that the digraph $D^*$ that we obtain by deleting $s$ and all its incident arcs and adding the arcs $u_1v_1,\ldots{},u_{d^-(s)}v_{d^-(s)}$ is $k$-arc-connected.
  \end{theorem}

  A digraph $D=(V,A)$ is {\bf minimally $\mathbf{k}$-arc-connected} if $\lambda{}(D)=k$ but $\lambda{}(D\setminus a)=k-1$ for every arc $a\in A$.

  \begin{theorem}[Mader~\cite{maderAM25}]
    \label{thm:minkas}
    Every minimally $k$-arc-connected digraph \jbj{has} a vertex $s$ with $d^-(s)=d^+(s)=k$.
  \end{theorem}

  
  
\FH{
Let $D$ be a digraph, and $s$ a new vertex (i.e. not in $V(D)$).
By \jbj{{\bf lifting} the arcs  $u_1v_1, \dots , u_kv_k$ to $s$} , we mean  adding $s$ and replacing these arcs by the $2k$ arcs $u_1s, \dots, u_ks, sv_1,\dots, sv_k$.

   \begin{lemma}[Mader~\cite{maderEJC3}]
    \label{lem:kbothways}
     Let $D$ be a $k$-arc-connected digraph, let $s$ be a new vertex,  let $u_1v_1, \dots , u_kv_k$ be $k$ distinct arcs, and let $D'$ be the digraph obtained from $D$ by lifting these arcs to $s$.  
     
     If $D$ is $k$-arc-connected, then $D'$ is also $k$-arc-connected.
     \end{lemma}
}

    We can now prove the following generalization of Proposition \ref{prop:3col-sub}. The reader is encouraged to compare this result with  Corollary \ref{cor:k+1ec}.

\begin{theorem}\label{thm:arc-con-bip}
  Every $k$-arc-connected digraph has a spanning $k$-arc-connected $(2k+1)$-partite subdigraph. This is best possible, as there are infinitely many $k$-arc-connected digraphs with no spanning $k$-arc-connected $2k$-partite subdigraph.
\end{theorem}
\begin{proof}
  The proof is by induction on the number of vertices of the digraph, the result holding trivially when there are at most $2k+1$ vertices.  So assume $D$ is a digraph on $n\geq 2k+2$ vertices and that the theorem holds for all $k$-arc-connected digraphs on $n-1$ vertices.
  Clearly we can assume that $D$ is  minimally $k$-arc-connected (otherwise we just delete some arcs). By Theorem \ref{thm:minkas}, $D$ contains a vertex $s$ with $d^-(s)=d^+(s)=k$. Thus we can apply Theorem \ref{thm:madersplit} with $s$ as the special vertex. Let $D^*$ be the $k$-arc-connected digraph that we obtain and let $A(s)$ be the set of $k$ arcs that we added to $D\langle V(D)\setminus \{s\}\rangle$ to obtain $D^*$. By induction,
  $D^*$ has a spanning $k$-arc-connected subdigraph $H$ with $\chi(H)\leq 2k+1$. Let
  $(X_1,\ldots{},X_{\chi{}(H)})$ be a partition of $V(H)=V(D)\setminus \{s\}$ into $\chi(H)$ sets such that $H=D^*[X_1,\ldots, X_{\chi{}(H)}]$.
 \FH{ Let $D'$ be obtained from $H\cup A(s)$ (some of the arcs of $A(s)$ may lie inside a set $X_i$) by lifting the $k$ arcs of $A(s)$ to $s$ . By Lemma \ref{lem:kbothways}, $D'$ is $k$-arc-connected.} As $s$ has at most $2k$ neighbours in $UG(D')$, either  there is a set $X_i$ such that $s$ has no neighbours in $X_i$ or $s$ is adjacent to all of the sets $X_1,\ldots{},X_{\chi{}(H)}$. In the former case adding $s$ to $X_i$, we get a $(2k+1)$-colouring of $D'$, and in the former case
  $\chi(H)\leq 2k$ (as the degree of $s$ in $D$ is $2k$)  and $(X_1,\ldots{},X_{\chi{}(H)},\{s\})$ is a $(2k+1)$-colouring of $D'$. 
  Hence $D'$ is a spanning $k$-arc-connected $(2k+1)$-partite subdigraph of $D$.

\medskip

  It remains to prove the second part of the theorem. We shall prove that for every pair of positive integers $r,k$ with $r\geq k+1$ there exists a $k$-arc-connected digraph $D$  on $(2k+1)r$ vertices such that every spanning $k$-arc-connected subdigraph $D'$ of $D$ satisfies $\chi{}(D')\geq 2k+1$.
  
  \FH{Let $R_{2k+1}$ denote the rotative tournament on $2k+1$ vertices $v_0,v_1,\ldots{},v_{2k}$ where the out-neighbours of $v_i$ are the vertices $v_{i+1},v_{i+2},\ldots{},v_{i+k}$ (indices modulo $2k+1$).}
  
 Let $D_1,D_2,\ldots{},D_r$ be copies of  $R_{2k+1}$ and denote the vertices of $D_i$ by $v_{i,0},\ldots{},v_{i,2k}$ such that the out-neighbours of $v_{i,j}$ in $D_i$ are the vertices $v_{i,j+1},\ldots{},v_{i,j+k}$ (indices mod $2k+1$). Add the arcs of the $k$th power of the $r$-cycle $C=v_{1,0}v_{2,0}\ldots{}v_{r,0}v_{1,0}$ (that is there is an arc from $v_{i,0}$ to $v_{i+1,0},\ldots{},v_{i+k,0}$ (modulo $r$) for $i\in [r]$, the resulting digraph is denoted $D_{k,r}$, see Figure \ref{fig:lambda2ex}.

  \begin{figure}[hbtp]
    \begin{center}
      \tikzstyle{vertexB}=[circle,draw, minimum size=20pt, scale=0.6, inner sep=0.5pt]
      \begin{tikzpicture}[scale=0.4]
        \node (v0) at (2,8) [vertexB] {$v_{1,0}$};
        \node (v1) at (2,6) [vertexB] {$v_{1,1}$};
        \node (v2) at (2,4) [vertexB] {$v_{1,2}$};
        \node (v3) at (2,2) [vertexB] {$v_{1,3}$};
        \node (v4) at (2,0) [vertexB] {$v_{1,4}$};
        \draw[->, line width=0.03cm] (v0) to (v1);
        \draw[->, line width=0.03cm] (v1) to (v2);
        \draw[->, line width=0.03cm] (v2) to (v3);
        \draw[->, line width=0.03cm] (v3) to (v4);
        \draw[->, line width=0.03cm] (v4) to [out=145,in=-145] (v0);
        \draw[->, line width=0.03cm] (v0) to [out=-115,in=115](v2);
        \draw[->, line width=0.03cm] (v1) to [out=-115,in=115](v3);
        \draw[->, line width=0.03cm] (v2) to [out=-115,in=115] (v4);
        \draw[->, line width=0.03cm] (v3) to [out= 135,in=-135] (v0);
        \draw[->, line width=0.03cm] (v4) to [out=135,in=-135] (v1);

        \node (v01) at (7,8) [vertexB] {$v_{2,0}$};
        \node (v11) at (7,6) [vertexB] {$v_{2,1}$};
        \node (v21) at (7,4) [vertexB] {$v_{2,2}$};
        \node (v31) at (7,2) [vertexB] {$v_{2,3}$};
        \node (v41) at (7,0) [vertexB] {$v_{2,4}$};
        \draw[->, line width=0.03cm] (v01) to (v11);
        \draw[->, line width=0.03cm] (v11) to (v21);
        \draw[->, line width=0.03cm] (v21) to (v31);
        \draw[->, line width=0.03cm] (v31) to (v41);
        \draw[->, line width=0.03cm] (v41) to [out=145,in=-145] (v01);
        \draw[->, line width=0.03cm] (v01) to [out=-115,in=115](v21);
        \draw[->, line width=0.03cm] (v11) to [out=-115,in=115](v31);
        \draw[->, line width=0.03cm] (v21) to [out=-115,in=115] (v41);
        \draw[->, line width=0.03cm] (v31) to [out= 135,in=-135] (v01);
        \draw[->, line width=0.03cm] (v41) to [out=135,in=-135] (v11);
        
        \node (v02) at (12,8) [vertexB] {$v_{3,0}$};
        \node (v12) at (12,6) [vertexB] {$v_{3,1}$};
        \node (v22) at (12,4) [vertexB] {$v_{3,2}$};
        \node (v32) at (12,2) [vertexB] {$v_{3,3}$};
        \node (v42) at (12,0) [vertexB] {$v_{3,4}$};
        \draw[->, line width=0.03cm] (v02) to (v12);
        \draw[->, line width=0.03cm] (v12) to (v22);
        \draw[->, line width=0.03cm] (v22) to (v32);
        \draw[->, line width=0.03cm] (v32) to (v42);
        \draw[->, line width=0.03cm] (v42) to [out=145,in=-145] (v02);
        \draw[->, line width=0.03cm] (v02) to [out=-115,in=115](v22);
        \draw[->, line width=0.03cm] (v12) to [out=-115,in=115](v32);
        \draw[->, line width=0.03cm] (v22) to [out=-115,in=115] (v42);
        \draw[->, line width=0.03cm] (v32) to [out= 135,in=-135] (v02);
        \draw[->, line width=0.03cm] (v42) to [out=135,in=-135] (v12);
        
        \node (v03) at (17,8) [vertexB] {$v_{4,0}$};
        \node (v13) at (17,6) [vertexB] {$v_{4,1}$};
        \node (v23) at (17,4) [vertexB] {$v_{4,2}$};
        \node (v33) at (17,2) [vertexB] {$v_{4,3}$};
        \node (v43) at (17,0) [vertexB] {$v_{4,4}$};
        \draw[->, line width=0.03cm] (v03) to (v13);
        \draw[->, line width=0.03cm] (v13) to (v23);
        \draw[->, line width=0.03cm] (v23) to (v33);
        \draw[->, line width=0.03cm] (v33) to (v43);
        \draw[->, line width=0.03cm] (v43) to [out=145,in=-145] (v03);
        \draw[->, line width=0.03cm] (v03) to [out=-115,in=115](v23);
        \draw[->, line width=0.03cm] (v13) to [out=-115,in=115](v33);
        \draw[->, line width=0.03cm] (v23) to [out=-115,in=115] (v43);
        \draw[->, line width=0.03cm] (v33) to [out= 135,in=-135] (v03);
        \draw[->, line width=0.03cm] (v43) to [out=135,in=-135] (v13);

        \node (v04) at (22,8) [vertexB] {$v_{5,0}$};
        \node (v14) at (22,6) [vertexB] {$v_{5,1}$};
        \node (v24) at (22,4) [vertexB] {$v_{5,2}$};
        \node (v34) at (22,2) [vertexB] {$v_{5,3}$};
        \node (v44) at (22,0) [vertexB] {$v_{5,4}$};
        \draw[->, line width=0.03cm] (v04) to (v14);
        \draw[->, line width=0.03cm] (v14) to (v24);
        \draw[->, line width=0.03cm] (v24) to (v34);
        \draw[->, line width=0.03cm] (v34) to (v44);
        \draw[->, line width=0.03cm] (v44) to [out=145,in=-145] (v04);
        \draw[->, line width=0.03cm] (v04) to [out=-115,in=115](v24);
        \draw[->, line width=0.03cm] (v14) to [out=-115,in=115](v34);
        \draw[->, line width=0.03cm] (v24) to [out=-115,in=115] (v44);
        \draw[->, line width=0.03cm] (v34) to [out= 135,in=-135] (v04);
        \draw[->, line width=0.03cm] (v44) to [out=135,in=-135] (v14);

        \draw[->,line width=0.03cm] (v0) to (v01);
        \draw[->,line width=0.03cm] (v01) to (v02);
        \draw[->,line width=0.03cm] (v02) to (v03);
        \draw[->,line width=0.03cm] (v03) to (v04);
        \draw[->,line width=0.03cm] (v04) to [out=135, in=45](v0);
        \draw[->,line width=0.03cm] (v04) to [out=150,in=30](v01);
        \draw[->,line width=0.03cm] (v03) to [out=150,in=30] (v0);
        \draw[->,line width=0.03cm] (v0) to [out=20,in=160] (v02);
        \draw[->,line width=0.03cm] (v01) to [out=20,in=160](v03);
        \draw[->,line width=0.03cm] (v02) to [out=20,in=160] (v04);

    \end{tikzpicture}
    \caption{The digraph $D_{2,5}$.}\label{fig:lambda2ex}
    \end{center}
  \end{figure}
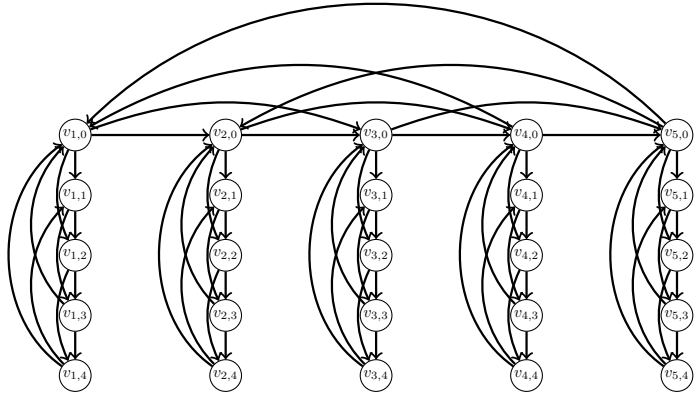

   Let us check that $\lambda{}(D_{k,r})=k$. Clearly $\lambda{}(D_{k,r})\leq k$ since we can disconnect it by removing the $k$ arcs from $D_1$ to $D_2$. 
  As the $k$th power of an $r$-cycle is  $k$-arc-connected  when $r\geq k+1$ (see e.g. \cite[Exercise 5.10]{bang2009}) and $R_{2k+1}$ is the $k$th power of a $(2k+1)$-cycle and hence also $k$-arc strong, it is easy to see that $d^+(S)\geq k$ for every non-trivial subset of $V(D_{k,r})$. Indeed, if both $S$ and $V\setminus S$ contain vertices from $D_i$ for some $i$, then there are at least $k$ arcs from $S\cap V(D_i)$ to $(V\setminus S)\cap V(D_i)$ (because $D_i$ is $k$-arc-connected)  so we may assume that $S$ contains all vertices of some but not all of $D_1,\ldots{},D_r$. Now it follows from the fact that the subdigraph formed by the vertices $v_{1,0}v_{2,0}\ldots{}v_{r,0}$ is $k$-arc-connected that $d^+(S)\geq k$. Hence $\lambda{}(D_{k,r})=k$.\\
  
We claim  that every spanning $k$-arc-connected subdigraph $D'$ of $D_{k,r}$ has $\chi{}(D)\geq 2k'+1$. This  follows by just looking inside $D_1$ and observing that, just to get the in- and out-degrees at least $k$, we need $2k+1$ colours: if we put any two vertices of $V(D_1)$ in the same set, at least one of these vertices  will have too few in- or out-neighbours.
\end{proof}

\section{Partitions maintaining high (out-)degrees}\label{sec:degrees}

In this section, we consider the problem of partitioning the vertex set of a digraph $D=(V,A)$ into $r$ sets $V_1,V_2,\ldots{},V_r$ such that the $r$-partite  digraph $D[V_1,V_2,\ldots{},V_r]$ satisfies certain minimum degree conditions.

The proof in \cite{kreutzerEJC24} of the following proposition is instructive of the techniques used in this section, so we include it for completeness.

\begin{proposition}[Kreutzer et al.~\cite{kreutzerEJC24}] \label{prop4k}
Every digraph has a majority $4$-colouring.
\end{proposition}

\begin{proof}
  Let $D$ be given and let $v_1,\ldots{},v_n$ be an arbitrary ordering of $V(D)$. For each $i=1,2,\ldots{},n$ in that order, assign to $v_i$ a first colour $c_1(v_i) \in [2]$ such that at least half of the out-neighbours of $v_i$ in $\{v_1,\ldots{},v_{i-1}\}$ have a first colour different from $c_1(v_i)$. Next we consider the reverse ordering $v_n,v_{n-1},\ldots{},v_1$ and, for each $j=n,n-1,\ldots{},1$ in that order, we assign to $v_j$ a second colour $c_2(v_j)\in [2]$ such that at least half of the out-neighbours of $v_j$ in $\{\jbj{v_{j+1}},\ldots{},v_n\}$ have a second colour different from $c_2(v_j)$. 
\FH{  Now let $c$ be the $4$-colouring defined by $c(v)=(c_1(v), c_2(v))$.
  This is clearly a majority colouring because the colour pair of $v_i$ differs from the ones of at least half of its out-neighbours in $\{v_1,\ldots{},v_{i-1}\}$ in the first coordinate, and from the ones of at least half of its out-neighbours in $\{v_{i+1},\ldots{},v_n\}$ in the second coordinate.}
  \end{proof}

Generalizations of the above result can also be found in \cite{antonio2017,knoxEJC25}.

\begin{theorem}\label{thm:2knotk+1}
There exists a \jbj{$2k$-strong} digraph with $\delta^+(D)=2k$ which has no spanning $3$-partite digraph $H$ with $\delta^0(H)\geq k+1$.
\end{theorem}
\begin{proof}
Let ${\cal D}$ be the class of  4-partite digraphs with partite sets $V_1,V_2,V_3,V_4$ that one can construct as follows: start with an independent set $V_4$ with $|V_4|=3k-2$.
For every subset $W \subseteq V_4$ of size $k$, we let $R_W$ be a  set of $2k-1$ vertices. Let $V_3$ be the union of all such $R_W$.  Hence $|V_3|=(2k-1)\times{3k-2 \choose k}$.  Now add all possible arcs from $R_W$ to $W$ as well as \jbj{extra} arcs from $V_3$ to $V_4$ such that all vertices in $V_3$ have out-degree exactly $2k$.
   Next, for each $W \subseteq V_4$ of size $k$ and for each $U \subseteq R_W$ also of size $k$, let $S_{W,U}$ and  $T_{W,U}$ be two sets of respectively $k$  and $2k$ new vertices, and add all possible arcs from  $S_{W,U}$ to $W \cup U$, all possible arcs from $T_{W,U}$ to $S_{W,U} \cup U$, and all possible arcs from $W$ to $T_{W,U}$.  Let $V_2$ be the union of all such $S_{W,U}$ and let $V_1$ be the union of all such $T_{W,U}$. Furthermore  add extra arcs from $V_4$ to $V_1$ such that each vertex in $V_1$ has in-degree exactly $2k$. Finally add all possible arcs from $V_4$ to $V_2 \cup V_3$. See Figure \ref{fig:2k}.

\begin{figure}[H]
\begin{center}
\tikzstyle{vertexSC}=[circle,draw, minimum size=25pt, scale=1, inner sep=0.1pt]
\tikzstyle{vertexLC}=[circle,draw, minimum size=40pt, scale=1, inner sep=0.1pt]
\begin{tikzpicture}[scale=0.9]
\draw (0,2) rectangle (2,6);
\node (w) at (1,4.5) [vertexSC] {$W$};
\node () at (0.5,3) {$V_4$};
\draw (4,2) rectangle (6,6);
\node (rw) at (5,4.5) [vertexLC] {};
\node (u) at (5,4.5) [vertexSC]{$U$};
\node () at (5,5.5) {$R_W$};
\node () at (5.5,3) {$V_3$};
\draw (8,2) rectangle (10,6);
\node (swu) at (9,4.5) [vertexSC] {$S_{W,U}$};
\node () at (9.5,3) {$V_2$};
\draw (3,-2) rectangle (7,0);
\node (twu) at (5,-1) [vertexLC] {$T_{W,U}$};
\node () at (6.5,-1.5) {$V_1$};
\draw [->,line width=0.03cm, double] (rw) to (w);
\draw [->,line width=0.03cm] (4.5,3) to (1.5,3);
\draw [->,line width=0.03cm, double] (swu) to (u);
\draw [->,line width=0.03cm, double] (swu) [out=120,in=60] to (w);
\draw [->,line width=0.03cm, double] (twu) to (u);
\draw [->,line width=0.03cm, double] (twu) to (swu);
\draw [->,line width=0.03cm, double] (w) to (twu);
\draw [->,line width=0.03cm] (1,2.5) to (3.5,-1);
\draw [->,line width=0.03cm, double] (2,4) to (4,4);
\draw [->,line width=0.03cm, double] (2,2) to [out=-30,in=210] (8,2);
\end{tikzpicture}
\end{center}
\caption{An illustration of a 4-partite digraph from the class $\cal D$.}\label{fig:2k}
\end{figure}
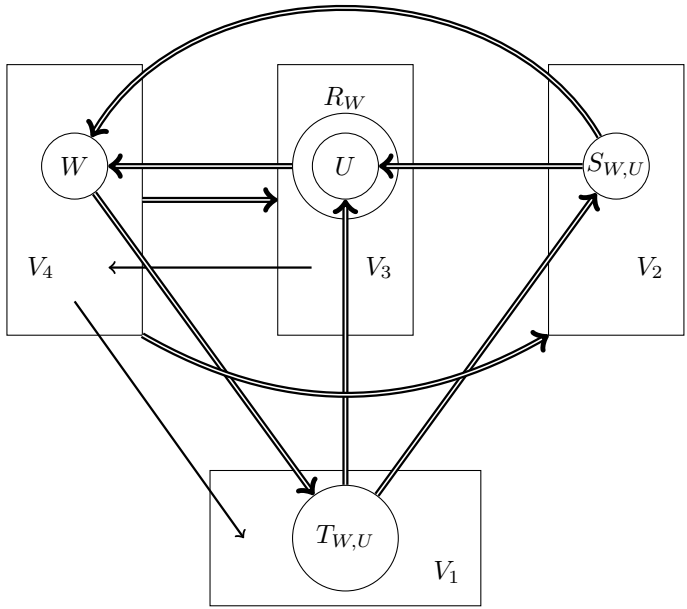

  Note that $|V_2|=k\times{2k-1 \choose k}\times{3k-2 \choose k}$ and $|V_1|=2k\times {2k-1 \choose k}\times {3k-2 \choose k}$.
  Also note that the out-degree of any vertex in $V_1 \cup V_2 \cup V_3$ is exactly $2k$. The out-degree of each vertex in $V_4$ is huge.
The in-degree of each vertex in $V_1$ is $2k$, the in-degree of each vertex in $V_2$ is $5k-2$  and the in-degrees of vertices in $V_3$ and $V_4$ are big.\\

\begin{claim} Every digraph $D\in {\cal D}$ is \jbj{$2k$-strong}.
\end{claim}
\begin{subproof}
 For the sake of contradiction, assume that \jbj{some $D\in \cal D$ has a separating set $S$   of size less than $2k$}.
    Let $D'=D-S$ and let $V_i' = V_i \setminus S$ for $i\in [4]$. Let $Z= V_2' \cup V_3'$ and let $z \in Z$ be arbitrary. As $z$ has $2k$ arcs into $V_4$ in $D$ we note that $z$ has an arc, say $zv$, into $V_4'$. Note that $v$ dominates all vertices in $Z$, by our construction of $D$. This implies that there is a path from $z$ to every vertex in $Z$ in $D'$. Therefore (as $z$ was chosen arbitrarily) all vertices in $Z$ belong to the same strong component of $D'$.     
    Let $v_4 \in V_4'$ be arbitrary. As $v_4$ has (a lot) more than $2k$ arcs entering it from $V_2 \cup V_3$ in $D$ it has at least one arc into it from $Z$. As it also has an arc into every vertex of $Z$ we note that it belongs to the same strong component as $Z$ in $D'$.  As $v_4$ was chosen arbitrarily we note that $Z \cup V_4'$ ($= V_2' \cup V_3' \cup V_4'$) belongs to the same strong component in $D'$. As $V_1$ is an independent set with $d^+(r)=d^-(r)=2k$ for every $r \in V_1$ we note that every vertex in $V_1'$ has an arc into $Z \cup V_4'$ and an arc from $Z \cup V_4'$ in $D'$.  Therefore there is only one strong component in $D'$ and $S$ was not a seperating set. This contradiction implies the claim. \end{subproof}

  We will now show that every $3$-partition of the vertices in $D$ contains a vertex $x$ such that $x$ has at most $k$ out-neighbours in different partite sets than itself or has  at most $k$ in-neighbours in different partite sets than itself.  \FH{Assume for a contradiction that there is a $3$-partition $(P_1, P_2, P_3)$ such that that \jbj{every vertex has} at least $k+1$ out-neighbours and at least $k+1$ in-neighbours in different partite sets than the vertex itself.
  As $|V_4|=3k-2$, one of the $P_i$, say $P_1$, contains a set $W$ of $k$ vertices of $V_4$.  } This implies that $R_W \subseteq P_2 \cup P_3$ (otherwise \jbj{some} vertex $x \in P_1 \cap R_W$ does not have $k+1$ out-neighbours in $P_2 \cup P_3$, a contradiction). Thus, as $|R_W|=2k-1$, one of its subsets $U$ of size $k$ is contained in one of the two partite sets $P_2$ and $P_3$. W.l.o.g. $U \subseteq P_2$.
  This implies that $S_{W,U} \subseteq P_3$ as every vertex in $S_{W,U}$ has out-degree exactly $2k$ and it has $k$ arcs into $P_1$ and $k$ arcs into $P_2$. 
  Let $t \in T_{W,U}$ be arbitrary. Note that  the out-degree of $t$ \jbj{in $D$} is $2k$ and $t$ has $k$ arcs into $P_2$ (in $R_W$) and $k$ arcs into $P_3$ (in $S_{W,U}$). Therefore $t \in P_1$. However this implies that $t$ does not have $k$ in-neighbours in $P_2 \cup P_3$, as it has in-degree $2k$ and $k$ of these arcs come from $W$ (and $W \subseteq P_1$). 
  This contradiction completes the proof.
  \end{proof}

If we want to keep the same minimum out-degree in a low-chromatic spanning subdigraph \jbj{$D'$ of a digraph $D$}, we can give the \jbj{best possible} bound on the chromatic number of such a digraph $D'$. \jbj{If  a digraph $D$ is not strongly connected, then a {\bf terminal strong component} of $D$ is a strong component $S$ such that $D$ has no arc $uv$ with $u\in V(S)$ and $v\in V\setminus V(S)$.}

\begin{proposition}[Bang-Jensen et al.~\cite{bangTCS719}]
    Every digraph $D$ with $\delta^+(D)\geq k$ has a spanning $(2k+1)$-partite subdigraph $D'$ with $\delta^+(D')\geq k$.
  \end{proposition}

  \begin{theorem}[Bang-Jensen et al.~\cite{bangTCS719}]
    \label{thm:2kcolpol}
    A digraph $D$ with $\delta^+(D)\geq k$ has a spanning $2k$-partite  subdigraph $D'$ with $\delta^+(D')\geq k$ if and only if no terminal strong component of $D$ is a $k$-regular tournament. In particular, there is a polynomial-time algorithm for deciding whether a given digraph has spanning  $2k$-partite subdigraph $D'$ with $\delta^+(D')\geq k$. 
    \end{theorem}

For the proofs below we need some lemmas.

\begin{lemma}[Alon~\cite{alonCPC15}]
\label{alonI}
 Let $A=(a_{ij})$ be an $n \times n$ matrix where $a_{ii}=0$ for all $i$ and $a_{ij} \geq 0$ for all $i \not= j$ and
$\sum_{j=1}^n a_{ij} \leq 1$ for all $i \in [n]$. Let $t>0$ be an integer and let $c_1,c_2,\ldots,c_t$ be positive real numbers
that sum up to $1$. Then there exists a partition $(S_1,S_2,\ldots,S_t)$ of $[n]$ such that for 
every $r$ and every $i \in S_r$ we have $\sum_{j \in S_r} a_{ij} \leq 2 c_r$.
\end{lemma}

\begin{lemma}
\label{lem:3part}
  Let $D$ be a digraph, and let $(X,Y)$ be an arbitrary partition of $V(D)$ into two disjoint subsets (one of which may be empty).
  $D$ has a spanning $3$-partite subdigraph $H$
  such that every vertex $x \in X$ satisfies $d^+_{H}(x)\geq \left\lceil\frac{d^+_D(x)}{3} \right\rceil$ and
  every vertex $y \in Y$ satisfies $d^-_{H}(y)\geq \left\lceil\frac{d^-_D(y)}{3} \right\rceil$.
\end{lemma}

\begin{proof}
The proof follows similar lines as in \cite{antonio2017}.
Let $D$ be a digraph, let $(X,Y)$ be an arbitrary partition of $V(D)=\{v_1,v_2,\ldots,v_n\}$ and let $A=(a_{ij})$ be the $n \times n$ matrix defined as follows.
For all $i \in [n]$ let $a_{ii}=0$. 
For each $v_i \in X$, let $a_{ij} = \frac{1}{d^+_D(v_i)}$ if $v_iv_j \in A(D)$, and let $a_{ij}=0$ if $v_iv_j \not\in A(D)$.
For each $v_i \in Y$ let $a_{ij} = \frac{1}{d^-_D(v_i)}$ if $v_jv_i \in A(D)$ and let $a_{ij}=0$ if $v_jv_i \not\in A(D)$.
We now apply Lemma~\ref{alonI} to $A$ with $t=3$ and $c_1=c_2=c_3=\frac{1}{3}$ in order to obtain
a partition $(S_1,S_2,S_3)$ of $[n]$ as in the lemma. Let $T_j = \{ v_i \; | \; i \in S_j\}$ for $j\in[3]$ and let $H=D[T_1,T_2,T_3]$. 
Note that the following holds for every $r \in [3]$ and every $v_i \in T_r$, by Lemma~\ref{alonI}.

\[
\frac{2}{3} = 2 c_r \geq \sum_{j \in S_r} a_{ij} = \sum_{v_j \in T_r} a_{ij} 
\]

So, if $v_i \in X$ then $\frac{2}{3} \geq \frac{d_{T_r}^+(v_i)}{d^+_D(v_i)}$, which implies that at most two thirds of $v_i$'s 
out-neighbours lie in $T_r$.  Therefore  $d^+_{H}(v_i)\geq \left\lceil\frac{d^+_D(v_i)}{3} \right\rceil$, as desired.
If $v_i \in Y$ then $\frac{2}{3} \geq \frac{d_{T_r}^-(v_i)}{d^-_D(v_i)}$, which implies that at most two thirds of $v_i$'s
in-neighbours lie in $T_r$.  Therefore  $d^-_{H}(v_i)\geq \left\lceil\frac{d^-_D(v_i)}{3} \right\rceil$, as desired.
\end{proof}

We now consider the case where we want to maintain a high minimum semi-degree by partitioning the vertex set into few parts.

Using a similar proof as that of Proposition~\ref{prop4k}, but considering both out- and in-neighbours (and using two 3-colourings, instead of two 2-colourings), one can easily prove that every digraph $D$ has a spanning $9$-partite subdigraph $H$ with $\delta^0(H)\geq\delta^0(D)/2$.
We can in fact improve this result as follows.

\begin{theorem}
\label{thm:KsemiDegrees}
  Let $k \geq 5$ be an integer. 
  Every digraph $D$ has a spanning $k$-partite subdigraph $H$ with $d^0_H(x)
  \geq \frac{k-4}{k} d^0_D(x)$ for all $x \in V(D)$.
\end{theorem}

\begin{proof}
Let $V(D)=\{v_1,v_2,\ldots,v_n\}$ and let $A=(a_{ij})$ be the $n \times n$ matrix defined as follows.

\[
a_{ij} = \left\{
\begin{array}{ccl} \vspace{0.15cm}
0 & \hspace{0.5cm} & \mbox{if $v_i v_j \not\in A(D)$ and $v_j v_i \not\in A(D)$} \\ \vspace{0.15cm}
\frac{1}{2d^+_D(v_i)} & & \mbox{if $v_i v_j \in A(D)$ and $v_j v_i \not\in A(D)$} \\ \vspace{0.15cm}
\frac{1}{2d^-_D(v_i)} & & \mbox{if $v_i v_j \not\in A(D)$ and $v_j v_i \in A(D)$} \\
\frac{1}{2d^+_D(v_i)} + \frac{1}{2d^-_D(v_i)} & & \mbox{if $v_i v_j \in A(D)$ and $v_j v_i \in A(D)$} \\
\end{array}
\right.
\]

Note that, as there are no loops in $D$, $a_{ii}=0$ for all $i \in [n]$.
Also $\sum_{j=1}^n a_{ij} \leq 1$ for all $i \in [n]$ since we sum up the term $\frac{1}{2d^+_D(v_i)}$ exactly $d^+_D(v_i)$ times and
we sum up the  term $\frac{1}{2d^-_D(v_i)}$ exactly $d^-_D(v_i)$ times (if $d^+_D(v_i)=0$ (resp. $d^-_D(v_i)=0$), then the first (resp. second) sum equals $0$.).
We now apply Lemma~\ref{alonI} to $A$ with $t=k$ and $c_1=c_2= \cdots = c_k=\frac{1}{k}$ in order to obtain
a partition $(S_1,S_2, \ldots, S_k)$ of $[n]$. Let $T_j = \{ v_i \; | \; i \in S_j\}$ for $j\in [k]$, and let $H=D[T_1, \ldots, T_k]$.
Note that the following holds for every $r \in [k]$ and every $v_i \in T_r$, by Lemma~\ref{alonI}.

\[
\frac{2}{k} = 2 c_r \geq \sum_{j \in S_r} a_{ij} = \sum_{v_j \in T_r} a_{ij}
\]

In the last sum, we note that the term $\frac{1}{2d^+_D(v_i)}$ is summed up exactly $d_{T_r}^+(v_i)$ times and the 
term $\frac{1}{2d^-_D(v_i)}$ is summed up exactly $d_{T_r}^-(v_i)$ times. Therefore the following holds.

\[
\frac{2}{k}  \geq \sum_{v_j \in T_r} a_{ij} = \frac{d_{T_r}^+(v_i)}{2d^+_D(v_i)} + \frac{d_{T_r}^-(v_i)}{2d^-_D(v_i)}
\]

Multiplying both sides by $2$, we obtain $\frac{4}{k} \geq \frac{d_{T_r}^+(v_i)}{d^+_D(v_i)} + \frac{d_{T_r}^-(v_i)}{d^-_D(v_i)}$
and as all terms are non-negative, we have $\frac{d_{T_r}^+(v_i)}{d^+_D(v_i)} \leq \frac{4}{k}$ and 
$\frac{d_{T_r}^-(v_i)}{d^-_D(v_i)} \leq \frac{4}{k}$. 
As $H=D[T_1,\ldots,T_k]$, we note that
$d^+_{H}(v_i)\geq (1 - 4/k) d^+_D(v_i)$ and 
$d^-_{H}(v_i)\geq (1 - 4/k) d^-_D(v_i)$, which completes the proof. 
\end{proof}

\jbj{\begin{corollary}Every digraph $D$ with $\delta^0(D)\geq 3r$ has a spanning 6-partite subdigraph $H$ with\\ $\delta^0(H)\geq r$.
    \end{corollary}}

\section{Final remarks and open questions}\label{sec:open}

\subsection{Spanning $k$T-subgraphs with low chromatic number}

As noted in the introduction, we do not know whether the edge-connectivity condition $\lambda(G)\geq 5$ in Corollary~\ref{cor:5ec2T3part}~(a) is best possible
\begin{question}
Does every $4$-\jbj{edge}-connected graph have a spanning  $3$-partite 2T-subgraph ?
\end{question}

Just as a 2T-graph is the union of two edge-disjoint spanning trees a {\bf $\mathbf{k}$T-graph} is a graph whose edge set decomposes into $k$ edge-disjoint spanning trees. 
Observe that $k$T-graphs are $(2k-1)$-degenerate and thus $2k$-partite. Moreover, there 
 there are infinitely many $k$T-graphs with chromatic number $2k$:
 For $k=1$ this is trivial and for $k=2$ the family of odd wheels $W_{2n+1}$, $n\geq 3$ forms an infinite class of 4-chromatic 2T-graphs. A well-known result due to Walecki states that the complete graph $K_{2r}$ decomposes into $r$ hamiltonian paths (see e.g. \cite{alspachBICA52}). Hence if we take an arbitrary $k$T-graph $H$ and join it by $k$ edges to a copy of $K_{2k}$, then we obtain a
$k$T-graph with chromatic number $2k$.\\

It would be interesting to study the existence of spanning $k$T-subgraphs with \jbj{chromatic number smaller than $2k$} in graphs with \jbj{sufficiently}   high edge-connectivity.
Corollary \ref{cor:5ec2T3part} can be generalized as follows.
\begin{proposition} Let $G$ be a graph and $k$ a positive integer.
If $\lambda(G) \geq 2k+1$, then $G$ contains a spanning $(k+1)$-partite $k$T-graph.
\end{proposition}
\begin{proof}
Let $G$ be $(2k+1)$-edge-connected. By Theorem \ref{thm:maxcut}, $G$ has a spanning $2k$-edge-connected graph $H$ with $\chi{}(H)\leq k+1$. By Theorem \ref{thm:2k}, $H$ has a spanning $k$T-graph.
\end{proof}

A natural question is whether edge-connectivity condition $\lambda(G)\geq 2k+1$ of this proposition is best possible.

\subsection{Highly arc-connected subdigraphs of low chromatic number}

The bound $2k+1$ of Theorem~\ref{thm:arc-con-bip} is best possible for digraphs in general.
However, for semicomplete digraphs of sufficiently large order and $k=1$, it is not as \jbj{we} established in Theorem~\ref{thm:SDstrongbip}.

\begin{question}
  Does every 2-arc-connected semicomplete digraph with at least 6 vertices have a spanning 2-arc-connected subdigraph $D'$ with $\chi{}(D')\leq 3$? \jbj{or} with $\chi{}(D')\leq 4$?
  
  More generally, does every $k$-arc-connected semicomplete digraph of sufficiently large order  have a spanning $k$-arc-connected $r$-partite subdigraph for $3\leq r \leq 2k$  ?
  \end{question}

 As established in Theorem~\ref{thm:nobipstrong},  there is no degree of strong connectivity which guarantees that a digraph contains a spanning strong bipartite subdigraph. But does there exist one that guarantees the existence of an $r$-partite subdigraph with high arc-connectivity \jbj{when $r\geq 3$.}
 \begin{question}
   \label{quest:g(k,r)}
  Does there exist a function $g=g(k,r)$  when $r\geq 3$ such that every  $g(k,r)$-arc-connected digraph has a spanning $r$-partite  subdigraph which is $k$-arc-connected?
\end{question}

By Proposition \ref{prop:3col-sub}, $g(1,3)=1$. 

\medskip

  An {\bf out-branching} is a connected digraph $B^+_s$ in which every vertex, except one (called the {\bf root}) vertex $s$  has exactly one arc entering. This is equivalent to saying that $s$ can reach every other vertex by a directed path in $B^+_s$. 

  The following classical result, due to Edmonds, 
  and Menger's theorem implies that every $k$-arc-connected digraph has $k$-arc-disjoint out-branchings rooted at $s$ for every vertex $s$.
\begin{theorem}[Edmonds~\cite{edmonds1973}]
  \label{thm:Edbrthm}
  Let $D=(V,A)$ be a digraph and let $s\in V$. Then $D$ contains  $k$ arc-disjoint out-branchings, all rooted at $s$, if and only if there are $k$ arc-disjoint $(s,v)$-paths in $D$ for every $v\in V$.
\end{theorem}

The following question concerning arc-disjoint out-branchings is certainly an important partial question in connection with Question  \ref{quest:g(k,r)}.

\begin{question}
  Does there exist a function $h=h(k,r)$ such that every  $h(k,r)$-arc-connected digraph has a spanning $r$-partite subdigraph with $k$ arc-disjoint out-branchings?
\end{question}

Note that we have $h(1,2)=1$ as every strong digraph has an out-branching.

Again the result of Theorem \ref{thm:nobipstrong}  makes even the following question relevant.
We saw in Theorem \ref{thm:7ecbip2T} that for undirected graphs and edge-disjoint spanning trees the corresponding edge-connectivity requirement is 7.

\begin{question}
Does $h(2,2)$ exist? 
\end{question}

By Theorem \ref{thm:7ecbip2T}, there exist infinitely many 6-edge-connected graphs with no spanning bipartite 2T-subgraph. 
\FH{A theorem of Nash-Williams~\cite{nashwilliamsCJM12} states that a graph has a $k$-arc-connected orientation if and only if it is $2k$-edge-connected.}
Thus, there exist infinitely many 3-arc-connected oriented graphs for which no  spanning bipartite subdigraph has a pair of arc-disjoint branchings. So $h(2,2)\geq4$ if it exists.

\subsection{Partitions maintaining high (out-)degrees}

Theorem~\ref{thm:KsemiDegrees} only holds for $k\geq 5$, so a natural question is whether it holds for smaller value
of $k$. Thomassen's examples show that it does not hold for $k=2$. But what for $k=3$ or $k=4$?
  \begin{problem} \label{ProbXi}
For $k\in \{3,4\}$ does there exists $c_k>0$, such that every digraph $D$ contains a 
spanning $k$-partite subdigraph $H$ such that every vertex $v$ satisfies $\delta_{H}^0(v)\geq c_k \delta_{D}^0(v)$ ?
\end{problem}

The answer to Problem~\ref{ProbXi} might be negative, yet a huge semi-degree in $D$ could still guarantee the existence of a spanning $3$- or $4$-partite subdigraph with large semi-degree.

 \begin{problem}
 For $k\in \{3,4\}$, does there exist $f_k(p)$ such that every digraph $D$ with $\delta^0(D)\geq f_k(p)$ has a spanning $k$-partite subdigraph $D'$ with $\delta^0(D')\geq p$?
\end{problem}

By Theorem \ref{thm:2knotk+1},  we must have $f_4(p)\geq f_3(p)\geq 2p-1$.

\medskip

Observe that $\frac{k-4}{k}$ tends to  $1$ as $k$ tends to $+\infty$. So for every $\epsilon$, there exists $k_{\epsilon}$ such that
Every digraph $D$ contains a spanning $k_{\epsilon}$-partite subdigraph $H$ such that every vertex $v$ satisfies  
$\delta_{H}^0(v)\geq (1-\epsilon) \delta_{D}^0(v)$. It is natural to ask for a minimum such $k_\epsilon$.

  \begin{problem} \label{ProbXii}
What is the smallest $k_{\epsilon}$ such that every digraph $D$ contains a 
spanning $k_{\epsilon}$-partite subdigraph $H$ such that every vertex $v$ satisfies $\delta_{H}^0(v)\geq (1-\epsilon) \delta_{D}^0(v)$ ?
\end{problem}

By Theorem \ref{thm:2knotk+1}, we have $k_{\epsilon} \leq \frac{4}{\epsilon}$. 
Note that Problem~\ref{ProbXii} is equivalent to determining (for all $k$) the largest $c_k$ such that every digraph $D$ has a spanning $k$-partite subdigraph $H$ with $d^0_H(x)
  \geq c_{k} d^0_D(x)$ for all $x \in V(D)$.


\end{document}